\documentclass[11pt]{amsart}
\usepackage{geometry}                
\usepackage{graphicx}
\usepackage{amssymb}
\usepackage{epstopdf}
\usepackage{xcolor}
\usepackage{hyperref}

\hypersetup{colorlinks=true,linkcolor=blue,urlcolor=blue,citecolor=red}

\let\originalleft\left
\let\originalright\right
\renewcommand{\left}{\mathopen{}\mathclose\bgroup\originalleft}
\renewcommand{\right}{\aftergroup\egroup\originalright}

\newtheorem{theorem}{Theorem}[]

\newtheorem{proposition}[theorem]{Proposition}

\newcommand{\R}{\mathbb{R}}
\newcommand{\s}{\mathcal{S}}
\newcommand{\F}{\mathcal{F}}

\DeclareMathOperator{\diag}{diag}

\title{A diffusion model for time-dependent compositional data}

\author{Lu Chen}
\address{Department of Mathematical Sciences, Kent State University, Kent, OH 44022, U.S.A.}
\email{chenlu81@hotmail.com}
\author{Omar De la Cruz Cabrera}
\address{Department of Mathematical Sciences, Kent State University, Kent, OH 44022, U.S.A.}
\email{odelacru@kent.edu}
\thanks{This research was supported in part by NSF grant DMS-1720259}
\author{Oana Mocioalca}
\address{Department of Mathematical Sciences, Kent State University, Kent, OH 44022, U.S.A.}
\email[Corresponding author]{omocioal@kent.edu}

\begin{document}
\maketitle

\begin{abstract}
We introduce a stochastic process for modeling the evolution in time
of compositional measurements (i.e., a vector of non-negative values that add up to a total of 1). This model
is a diffusion, as it is defined as the solution for a stochastic differential equation 
in the It\^o sense, and it has a Dirichlet distribution as
its steady distribution. We have named this process \emph{Dirichlet Diffusion} (DD).

As the process is confined to a manifold and the coefficients of the equation
are not globally Lipschitz, the usual theorems do not apply directly
and establishing the existence and properties of solutions requires a somewhat delicate analysis. 
We establish
the existence of strong solutions under the assumption 
that all the parameters of the Dirichlet distribution are greater than 2; for the general
case we were only able to establish the existence of weak solutions, but also that these solutions remain
confined to the closed simplex without need for reflecting boundaries. 

A useful feature of DD that it inherits from the Dirichlet distribution is the property of
\emph{aggregation}: If components are combined to create a coarser composition, the resulting
process is also a DD. This makes it useful,
for example, for jointly modeling the evolution of a microbiome grouping the microbe species
at different taxonomic levels.
\end{abstract}

\section{Introduction}
\subsection{Background}
\emph{Compositional data} are measurements of the relative abundances of the
components of a mix, aggregate, or combination; they are used when the
absolute quantities are not important, not relevant, and/or difficult to measure;
see, e.g., \cite{JA82}.
Compositional data are usually recorded so that
each data point is a vector of non-negative numbers that add up to a fixed positive total,
usually 1. 
Examples are found in fields as diverse as ecology (e.g., species abundance of trees
in a forest), geology (e.g., mineral composition of rock samples), political science (e.g.,
voter preference), genomics (e.g., transcript abundance in gene expression studies), 
and many more.

Due to the constraint on the total,
compositional data tend to exhibit negative correlations among its components:
larger relative abundance of one component tends to go together with smaller
relative abundances for the others. The particular difficulties that arise in the 
analysis of compositional data have been recognized
at least since 1897 by Karl Pearson, and many others afterwards \cite{JA82}.

Sometimes compositional data are collected in time series. This is typically because
compositional changes in time are of interest, or because repeated measurements 
are required and it is not practical to obtain independent samples. In our main example,
the study of microbiome composition, both motives apply: 
the molecular methods used to measure the composition (that is, the relative
proportions of different types of microbes) of a microbial community are noisy,
rendering single observations difficult to interpret; at the same time, changes
over time are often of interest. 
Figure~\ref{fig:metagenomicFlagPlots} shows time series data extracted from a study 
in which human subjects had microbiome composition
measurements taken in different parts of the body almost daily over six months  
\cite{CaporasoEtAl2011}.
The high variability indicates that a single measurement, taken on a particular day,
would not provide a useful picture; nevertheless, measurements on consecutive days
tend to be correlated, suggesting that the composition changes in a continuous,
stochastic manner.

\begin{figure}
\includegraphics[scale=0.75]{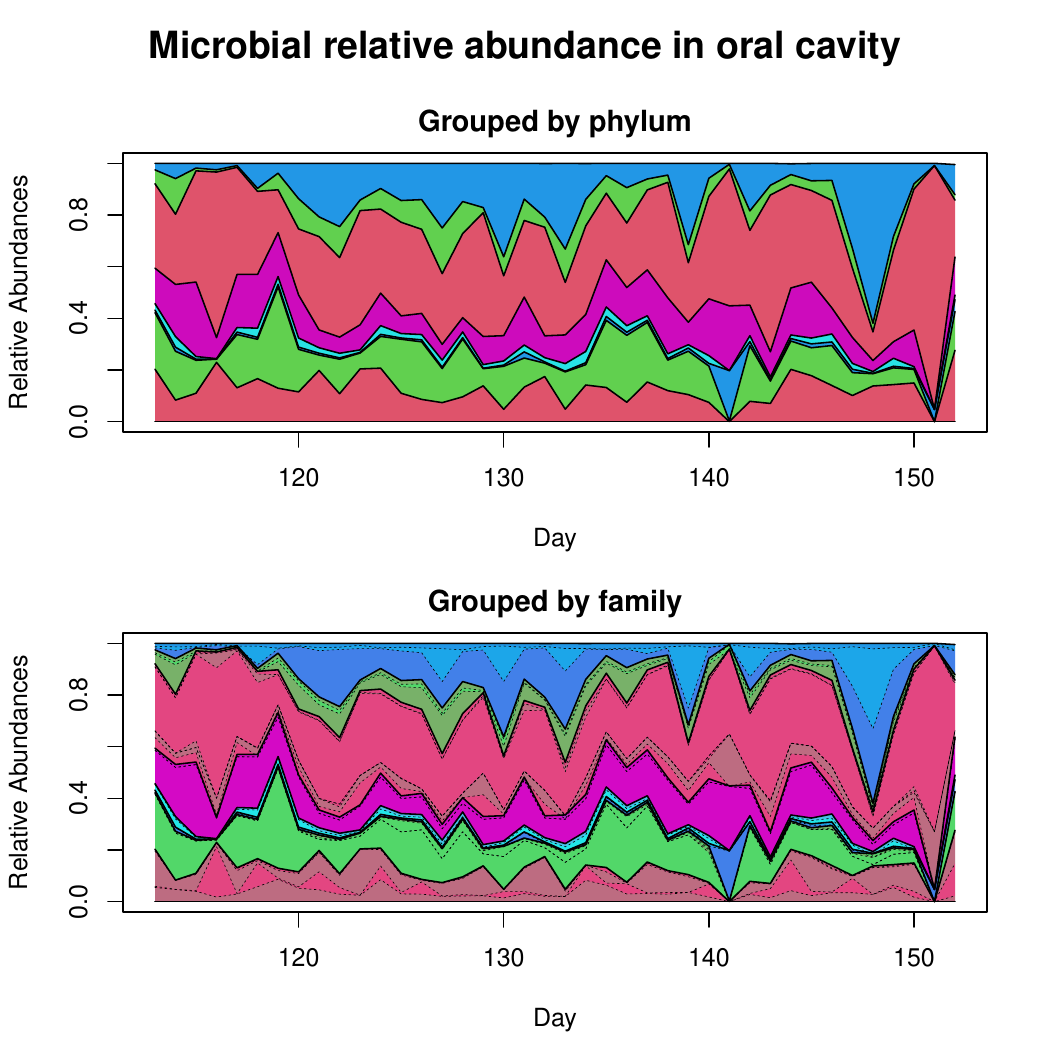}
\caption{Compositional time series data.
The plots show an uninterrupted 40-day stretch of microbiome composition measurements from the oral
cavity of individual F4,
from a study in which four human subjects
had measurements taken of the composition of the microbial
communities in different parts of the body almost daily over six months  \cite{CaporasoEtAl2011}.
The top panel shows compositions in which the microbe species
are grouped by phylum (17 different phyla); the bottom panel shows compositions in which the species
are grouped by family, a finer taxonomic rank (97 different families). Some categories are
too small to appear in the plots.
}\label{fig:metagenomicFlagPlots}
\end{figure}

In this article we introduce a continuous-time Markovian stochastic process $\left\{X(t)\right\}_{t\geq0}$,
which we call a \emph{Dirichlet Diffusion} (DD),
as a model for time-dependent
compositional data. 
It is obtained as the solution of a stochastic
differential equation (SDE); 
see Section~\ref{subsection:threeFormulations}
for the definition of the process. A key property is that it has a Dirichlet distribution as its invariant (or steady state)
probability distribution; thus, if $X(0)$ has that Dirichlet distribution,
each marginal $X(t)$ will also have that distribution; in practice, that would be the case, approximately, if
enough time has elapsed from the last perturbation of the system for it to reach its steady state.
In Section~\ref{sec:dirichletDistribution} we define the Dirichlet distribution and briefly discuss some of its properties.

\subsubsection*{Remarks about Notation}
Assume $d\geq2$ is fixed. We define the simplices $\s=\{x\in\R^d: x_1+\cdots+x_d=1\quad\mathrm{and}
\quad x_1,\dots,x_d>0\}$, and 
$\s^-=\{x\in\R^{d-1}:x_1+\cdots+x_{d-1}<1\quad\textrm{and}\quad x_1,\dots,x_{d-1}>0\}$. $\s^-$ is an open
subset of $\R^{d-1}$, while $\s$ is an open subset of a $(d-1)$-dimensional affine subspace of $\R^d$ (but it
is not open in $\R^d$). 
Also, we denote by $\mathcal{O}_+^d$ the open positive orthant $\{x\in\R^d:x_1,\dots,x_d>0\}$.
For simplicity, sometimes we will use the notation $S(x)=x_1+\cdots+x_d$.

For convenience, we also define two mappings, denoted by the superscripts ``$-$'' and ``$+$''. 
If $x=(x_1,\dots,x_d)\in\R^d$, then $x^-$ is defined as $(x_1,\dots,x_{d-1})\in\R^{d-1}$. If 
$x=(x_1,\dots,x_{d-1})\in\R^{d-1}$, then $x^+$ is defined as 
$\left(x_1,\dots,x_{d-1},1-\sum_{i=1}^{d-1}x_i\right)\in\R^d$. These two maps are bijections between
$\s$ and $\s^-$, and 
they are each other's inverses. We call $x^-$ the
\emph{truncation} or \emph{truncated version} of $x$.

Whenever matrix operations are used, vectors are assumed to be column vectors; the superscript $T$
denotes the matrix transpose. If $x$ is a vector, $\diag(x)$ is the diagonal matrix that
has the components of $x$ along its main diagonal.

\subsection{The Dirichlet distribution}\label{sec:dirichletDistribution}
Let $d\geq2$ and $\alpha=(\alpha_1,\dots,\alpha_d)$, with $\alpha_1,\dots,\alpha_d>0$, and define
$\alpha_0=\sum_{i=1}^d\alpha_i$.
The \emph{Dirichlet distribution} (see, e.g., \cite{kbj2000}) of order $d$ and
parameter vector $\alpha$
is a probability distribution for a random vector $X\in\R^d$ with density function proportional to
$\prod_{i=1}^{n} x_i^{\alpha_i-1}$, for $x\in\s$, and zero otherwise. 
However, this is a density only relative to the Lebesgue measure in $\R^{d-1}$ and a particular
one-to-one mapping between $\s$ and some open subset of $\R^{d-1}$.
Unfortunately, there is no canonical choice for such a parameterization of $\s$;
the truncation map defined above is a serviceable option, and the most common choice. Thus, we can
define the Dirichlet distribution as the distribution of a truncated
random vector $X^-=(X_1,\dots,X_{d-1})\in\R^{d-1}$ with support $\s^-$;
the missing component is implicitly taken to be the complement of the rest. Then
 \begin{equation}\label{eq:DirichletDensity}
 f(x^-;\alpha)= \frac{\Gamma (\alpha_0)}{\prod_{i=1}^d \Gamma(\alpha_i)} 
 \left(\prod_{i=1}^{d-1} x_i^{\alpha_i-1}\right)\left(1-\sum_{i=1}^{d-1} x_i \right)^{\alpha_d-1}\quad
 \textrm{for $x\in\s^-$,  zero otherwise}
 \end{equation}
is a proper density function. For convenience, we will abuse notation by regarding it as a density
for $X\in\R^d$, and may write $f(x)$ instead of $f(x^-)$.

The mean of $X\sim\mathrm{Dirichlet}(d,\alpha)$ is the vector $\mu=\alpha/\alpha_0$. An
alternative parameterization for the distribution is to use $\mu$ (or $\mu^-$) and $\alpha_0$; the
latter is called the \emph{concentration parameter}.

The Dirichlet distribution has the 
\emph{aggregation property}: If $(X_1,X_2,\dots,X_d)$ has distribution $\textrm{Dirichlet}(d,\alpha)$,
and  $\{B_1,B_2,\dots,B_k\}$ is a partition of the set $\{1,2,\dots,d\}$, then
$ \left(\sum_{i\in B_1} X_i,\sum_{i\in B_2} X_i,\dots,\sum_{i\in B_k}X_i\right)^T $ has
distribution
$\mathrm{Dirichlet}(k,\beta)$, where $\beta_j=\sum_{i\in B_j} \alpha_i$, for $j=1,\dots,k$.
The aggregation property is useful when modeling compositional data at different levels of detail, e.g., 
microbiome composition at the level of phylum, class, order, family, etc. 

For more details on the Dirichlet distribution, see, e.g.,~\cite{kbj2000,balakrishnan2004primer}

\section{Formulations of the SDE and Initial Value Problem}\label{subsection:threeFormulations}
We define the DD process as the 
solution of a particular SDE initial value problem. In 
this section we formulate three versions of the SDE, which we will later show 
are equivalent.

We will use SDEs of the form
\begin{equation}\label{eq:sde}
dX(t) = b(t,X)dt + \sigma(t,X)dW(t), 
\end{equation}
where $b(t,x)\in\R^d$ and $\sigma(t,x)\in\R^{d\times m}$ are defined on $[t_0,\infty)\times \bar{D}$,
with $D$ an open subset of $\R^d$. Given an initial condition $X(t_0)$, that is, a random vector
with values a.s.\ in $D$, a solution is a stochastic process $X(t)$ with values in $\bar{D}$ and indexed
by $t\in[t_0,\tau)$, where $\tau\leq\infty$ might be random. $W(t)$ is the standard Brownian motion
in $\R^m$, that is, $W_1(t),\dots,W_m(t)$ are independent instances of standard Brownian motion.
The SDEs considered in this article are time-homogeneous, i.e., $b$ and $\sigma$ do not depend
on $t$ (thus, we can assume $t_0=0$); however, we state Theorem~\ref{thm:khasminskii-corollary3-1}
(extracted from \cite{K12}) in the general case. SDEs in this paper are all in the It\^o sense; 
see, e.g., \cite{KS91} for a detailed exposition of the theory of Stochastic Analysis.

Fix $\theta>0$ and $\alpha=(\alpha_1,\dots,\alpha_d)^T$, with $\alpha_1,\dots,\alpha_d>0$. 
As above, call $\alpha_0=\sum_j\alpha_j$ and $\mu=\alpha/\alpha_0$.

\subsection{The canonical formulation.} Consider the SDE
\begin{equation} \label{canonicalSDE}
 dX=-\theta\left( {X}-\mu\right)dt+\sqrt{\frac{2\theta}{\alpha_0}}\diag\left(\sqrt{X}
\right)\left(I_d-\sqrt{X} \sqrt{X}^T\right)dW.
\end{equation} 
Here
$\sqrt{X}$ denotes the componentwise square root, and $\diag(v)$ is the diagonal matrix
that has the components of $v$ along its diagonal. 
The coefficients are defined throughout the open positive orthant of $\R^d$, but we will require
that the initial distribution satisfies $X(0)\in\s$ with probability 1. 

We will prove later that 
$X(t)$ remains in $\s$ for all $t\geq0$ (almost surely); thus, $\sqrt{X}$ is on the sphere of radius~1,
intersected with the open positive orthant, and $I_d-\sqrt{X} \sqrt{X}^T$ is the projection matrix
for the $(d-1)$-dimensional subspace of $\R^d$ of vectors orthogonal to $\sqrt{X}$.
Intuitively, $I_d-\sqrt{X} \sqrt{X}^T$ projects the increments of the $d$-dimensional
standard Brownian motion to move only in directions tangent to the sphere at the point $\sqrt{X}$;
after that, the components of the projected increments are rescaled by $\diag\left(\sqrt{X}\right)$,
making them smaller as $X$ approaches the boundaries.
The drift term corresponds to mean reversion, with rate $\theta$; since $\theta$ appears
also as a factor in the diffusion term, inside a square root, it can be interpreted as a time scale
factor.

\subsection{The adjusted  formulation.} Consider now the slightly modified SDE
\begin{equation} \label{adjustedSDE} 
dX=-\theta\left( {X}- \mu\right)dt+\sqrt{\frac{2\theta}{\alpha_0}}\diag\left(\sqrt{X}
\right)\left(S(X)I_d-\sqrt{X} \sqrt{X}^T\right)dW,
\end{equation} 
Here $S(X)=\sum_{i=1}^dX_j$.
Assuming as above that the initial distribution satisfies $X(0)\in\s$ with probability 1,
we will prove that, under some conditions, $X(t)$ remains in $\s$ for all $t\geq0$, meaning that $S(X)$
will be constantly equal to 1, making this formulation equivalent to the canonical formulation
\eqref{canonicalSDE}.
However, \eqref{adjustedSDE} has better properties
for $X$ outside of $\s$, leading to theoretical advantages in some proofs, as well as
better stability in simulations.

\subsection{The truncated formulation.} Here $X^-$ takes values in $\R^{d-1}$, while $W$ is
still a standard Wiener process in $\R^d$. Also, $\alpha^-=(\alpha_1,\dots,\alpha_{d-1})^T$,
while $\alpha_0$ still denotes $\alpha_1+\cdots+\alpha_d$. Thus, $\sigma(x)$ in this case
is a matrix of size $(d-1)\times d$.

\begin{equation} \label{truncatedSDE}
 dX^-=-\theta\left( {X^-}- \frac{\alpha^-}{\alpha_0}\right)dt+\sqrt{\frac{2\theta}{\alpha_0}}\diag\left(\sqrt{X^-}
\right)\left(I_d^--\sqrt{X^-} \sqrt{X^+}^T\right)dW
\end{equation} 
where $X^+\in\R^d$ contains all the components from $X^-$ plus one last component
equal to $1-\sum_{i=1}^{d-1} X_j$, and $I_d^-$ is obtained from $I_d$ by dropping the bottom row.
The initial distribution must satisfy the condition that $X^-(0)\in\s^-$ with probability 1.

\section{Existence of Solutions}

This section has two parts. First, we show that the equations of the previous section (with the
condition that $X(0)\in\s$ a.s.) have weak solutions that are equivalent and unique in distribution;
these solutions might hit the boundary of $\s$ but do not exit $\bar\s$ for any $t\geq0$ (a.s.).

In the second part we show that, if we add the condition that $\alpha_1,\dots,\alpha_d>2$, then
the solutions do not hit the boundary for any $t\geq0$ (a.s.). Furthermore, in this case we can
prove the existence of strong solutions.

\subsection{Existence of weak solutions}

Recall that a weak solution to an SDE like~\eqref{eq:sde} has the form 
$\left((\Omega,\F,P),\{\F_t\},W(t),X(t)\right)$,
where $(\Omega,\F,P)$ is a probability space, $\{\F_t\}$ is a filtration of sub-$\sigma$-fields of $\F$,
$W(t)$ is a Brownian motion adapted to $\{\F_t\}$, and $X(t)$ is a continuous process adapted
to $\{\F_t\}$, 
satisfying 
that $P\left[\int_0^t\{|b_i(s,X_s)|+\sigma^2_{ij}(s,X_s)\}\,ds<\infty \right] = 1$
for all $i=1,\dots,m$, $j=1,\dots,d$, and $t\geq0$,
as well as satisfying the integral form of the SDE, $P$-a.s. (see, e.g., \cite{KS91}).

\begin{theorem}\label{thm:existenceweaksolutions}
The SDEs \eqref{canonicalSDE}, \eqref{adjustedSDE}, and \eqref{truncatedSDE} have weak
solutions. Specifically:
\begin{enumerate}
\renewcommand{\theenumi}{\alph{enumi}}

\item Given a probability distribution on $\R^d$ with support contained in $\s$, there exist weak
solutions for SDEs  \eqref{canonicalSDE} and \eqref{adjustedSDE} such that $X(0)$ has the 
prescribed distribution, and $X(t)\in\bar\s$ for all $t\geq0$, $P$-a.s.
\item Given a probability distribution on $\R^{d-1}$ with support contained in $\s^-$, there exists a weak
solution for SDE \eqref{truncatedSDE} such that $X^-(0)$ has the 
prescribed distribution, and $X^-(t)\in\bar\s^-$ for all $t\geq0$, $P$-a.s.
\end{enumerate}
Furthermore, these solutions are unique in distribution, and they are equivalent, in the sense that:
\begin{enumerate}
\setcounter{enumi}{2}
\renewcommand{\theenumi}{\alph{enumi}}

\item Any weak solution for \eqref{canonicalSDE} is a solution for \eqref{adjustedSDE} and viceversa.
\item If $\left((\Omega,\F,P),\{\F_t\},W(t),X(t)\right)$ is a weak solution for \eqref{canonicalSDE} and \eqref{adjustedSDE},
then $((\Omega,\F,P),\allowbreak\{\F_t\},W(t),X^-(t))$ is a weak solution for \eqref{truncatedSDE}.
\item If $\left((\Omega,\F,P),\{\F_t\},W(t),X^-(t)\right)$ is a weak solution for \eqref{truncatedSDE},
then $((\Omega,\F,P),\{\F_t\},\allowbreak W(t),X^+(t))$ is a weak solution for \eqref{canonicalSDE} and \eqref{adjustedSDE}.
\end{enumerate}
\end{theorem}

The full proof of Theorem~\ref{thm:existenceweaksolutions} can be found in the Appendix.
Here is a short, informal, description of that proof: 
First we extend the drift and diffusion terms, which are continuous and bounded on
$\bar\s$, to continuous and bounded functions defined on $\R^d$, which guarantees the existence of
a weak solution $X(t)$, for $t\geq0$. Second, we prove that the solution satisfies $X(t)\in\bar\s$ for all $t\geq0$,
by the following argument: The first time (if ever) that $X(t)$ hits the boundary of $\bar\s$, the
component of the diffusion term orthogonal to the boundary becomes zero, while the drift term
``pulls'' the solution back towards the interior. Thus, $X(t)$ cannot escape $\bar\s$. Therefore,
since the extended drift and diffusion terms agree with the original ones on $\bar\s$, $\{X(t),t\geq0\}$
is a solution for the original SDE.

\subsection{Existence of strong solutions when $\min(\alpha_1,\dots,\alpha_d)>2$}\label{sec:existenceStrong}
\begin{proposition}
The SDEs (\ref{canonicalSDE}) and (\ref{adjustedSDE}), 
together
with the corresponding initial values, have unique strong solutions of the form 
$(X(t),0\leq t<\tau_1)$ and 
$(X(t),0\leq t<\tau_2)$, respectively,
in the open positive orthant
of $\R^d$.
The SDE (\ref{truncatedSDE})
has a unique strong solution of the form $(X^-(t),0\leq t<\tau_3)$ 
in the open set $\s^-\subset\R^{d-1}$. Here $\tau_1$, $\tau_2$, and $\tau_3$ are
stopping times w.r.t.\ the corresponding Wiener process, and they are maximal in the sense that the 
solutions cannot be extended beyond them (in the corresponding open set).
\end{proposition}
\begin{proof}
The coefficients are continuously differentiable, and therefore locally Lipschitz, on the corresponding open sets:
the open positive orthant, for SDEs~\eqref{canonicalSDE} and \eqref{adjustedSDE}, and 
$\s^-\subset\R^{d-1}$ for SDE~\eqref{truncatedSDE}; to treat all cases simultaneously, call such open set $U$.
Existence and uniqueness of strong solutions up to an explosion time follow from well known 
arguments, which we only briefly sketch here (these arguments can be found in, e.g., \cite{K12}, Section~3.4).

Let $U_r$, $r=1,2,\dots$ be a sequence of open sets such that for all $r$, $\bar{U}_r$ is
compact and contained in $U_{r+1}$, and $\bigcup_r\bar{U}_r=U$. For each $r$, the equation coefficients
can be modified so that they match the original coefficients within $U_r$, equal 0 outside $U_{r+1}$, yet their
derivatives exist, are continuous,  and are bounded
on all of $U$, so that the global Lipschitz and linear growth conditions are satisfied and therefore
a unique strong solution $X^{(r)}(t)$ exists (by classical theorems, e.g., \cite{KS91}, Theorem~5.2.9)
and has values in $U$ for $t\in[0,\infty)$. Let 
$\tau^{(r)}=\inf_{t}\{X^{(r)}(t)\notin U_r\}$; this is a stopping time, and $r\leq r'$ implies $\tau^{(r)}\leq\tau^{(r')}$
almost surely. Then $\tau := \lim_r \tau^{(r)}$ is a stopping time (the explosion time), and a solution $\{X(t):0\leq t < \tau\}$
can be defined so that for each $r$, $X(t)=X^{(r)}(t)$ for $0\leq t < \tau^{(r)}$. This solution is well defined and unique.
\end{proof}

\begin{proposition}\label{prop:solAdjustCanon}
If $(X(t),0\leq t<\tau)$ is a solution of (\ref{adjustedSDE}), where $\tau$ is a stopping time, 
and the initial condition $X(0)$ satisfies  $P[X(0)\in\s]=1$, then $(X(t),0\leq t<\tau)$ is a solution
of (\ref{canonicalSDE}).
\end{proposition}
\begin{proof}
We will show that $S(X(t))$ is constantly equal to 1, and thus since $X$ satisfies (\ref{adjustedSDE}),
it also satisfies (\ref{canonicalSDE}).

Let $S=S(X(t))$. Then, by Ito's formula, $S$ is the solution of an SDE with drift term $-\theta 1^T(X-\alpha/\alpha_0)dt$
(here $1^T = (1,\dots,1)$), which equals $\theta(1-S)dt$, and diffusion term given by:
\begin{align*}
\sqrt{\frac{2\theta}{\alpha_0}}1^T\diag\left(\sqrt{X}\right)\left(I_d-\sqrt{X}\sqrt{X}^T\right)dW &= 
\sqrt{\frac{2\theta}{\alpha_0}}\left(\sqrt{X}^T - \sqrt{X}^T\sqrt{X}\sqrt{X}^T \right)dW\\
 &=\sqrt{\frac{2\theta}{\alpha_0}}(1-S)\sqrt{X}^TdW,
\end{align*}
since $\sqrt{X}^T\sqrt{X} = S$. Then $M(t):=\int_0^t\sqrt{X}^TdW$ is a semimartingale, and $S$
satisfies the scalar SDE $dS = (1-S)dM$.
Clearly, $S$ constantly equal to 1 satisfies this equation and the initial condition $S(0)=1$;
by uniqueness, we conclude that $S\equiv 1$.
\end{proof}

Proposition~\ref{prop:solAdjustCanon} implies that the explosion time for the solution of (\ref{canonicalSDE})
is greater than or equal to the explosion time for the solution of (\ref{adjustedSDE}). The following
theorem proves that, if $\alpha_1,\dots,\alpha_d>2$, then the latter is $+\infty$, and therefore neither process leaves
the open positive orthant in finite time (a.s.).

\begin{theorem}\label{thm:noExplosionAdjusted}
Assume that $\alpha_1,\dots,\alpha_d>2$. Then the explosion time $\tau$ for the solution to (\ref{adjustedSDE})
is $+\infty$ (i.e., the process does not hit the boundaries), with probability~1.
\end{theorem}

The proof is based on verifying the conditions for regularity of solutions developed in Section~3.4 of~\cite{K12};
in particular, we will use Corollary~3.1 therein. For ease of reference, we will state here as a theorem the
parts of that corollary that are needed here.

\begin{theorem}[Khasminskii; part of Corollary 3.1 from \cite{K12}]\label{thm:khasminskii-corollary3-1}

Let $b(t,x)\in\R^d$ and $\sigma(t,x)\in\R^{d\times n}$ be continuous and defined on $[t_0,\infty)\times D$, where
$D$ is an open subset of $\R^d$. Let $A(t,x)=(a_{ij}(t,x))=\sigma(t,x)\sigma(t,x)^T$.
Assume that $D=\bigcup_k D_k$, where $\{D_k\}$ is a non-decreasing sequence
of open sets, with the closure of each $D_k$ contained in $D$, and that for each $k=1,2,\dots$, the
following conditions hold:
\begin{equation*}
|b(t,x)-b(t,y)|+\sum_{i,j}|\sigma_{ij}(t,x)-\sigma_{ij}(t,y)|\leq M_k |x-y| \\
\tag{Condition 1}
\end{equation*}
\begin{equation*}
|b(t,x)|+\sum_{i,j}|\sigma_{ij}(t,x)|\leq M_k\left(1+|x|\right),
\tag{Condition 2}
\end{equation*}
where the bound $M_k>0$ depends only on $k$. Assume also that there exists a function
$V(t,x)$ defined on $[t_0,\infty)\times D$ that is continuously differentiable w.r.t.\ $t$ and twice continuously
differentiable w.r.t.\ $x$, satisfying the conditions:
\begin{equation*}
LV \le cV,
\tag{Condition 3}
\end{equation*}
for some constant $c>0$, where $L$ is the differential operator defined by
$$
LV(t,x) = \frac{\partial V(t,x)}{\partial t} + 
\sum_i b_i(t,x)\frac{\partial V(t,x)}{\partial x_i} +
\frac{1}{2} \sum_{ij}a_{ij}(t,x)\frac{\partial^2 V(t,x)}{\partial x_i \partial x_j};
$$
and
\begin{equation*}
\inf_{t>0,x \in D\setminus D_k} V(t,x) \to \infty \hspace{20pt} as \hspace{5pt} k \to \infty.
\tag{Condition 4}
\end{equation*}
Then, given an instance $W(t)$ of standard $d$-dimensional Brownian motion, for every random
variable $X(t_0)$ with values almost surely in $D$, independent from $W(t)$, there exists a
strong solution $X(t)$, $t\geq t_0$, for the equation
$$
dX(t) = b(t,X)dt + \sigma(t,X)dW(t)
$$
that is Markovian, almost surely continuous, and satisfies
$$
P[X(t)\in D \textrm{ for all $t\geq t_0$}]=1.
$$ \qed
\end{theorem}

\begin{proof}[Proof of Theorem~\ref{thm:noExplosionAdjusted}]
Let 
$q=\sqrt{\min(\alpha_1,\dots,\alpha_d)/2}$;
by hypothesis, $q>1$. 
Consider the open set $D=\{x\in\R^d:  x_1,\dots,x_d>0 , S(x)<q\}$. 
Define a sequence of open sets 
$D_k:\{x=(x_1,\cdots, x_{n}): \text{each} \hspace{5pt} x_i>\frac{1}{k} \hspace{5pt} \text{and}\hspace{5pt} S(x) <q-\frac{1}{k}\}$,
$k=2,3,\cdots$. Then $\bigcup D_k =D$, and the closure of each $D_k$ is contained in $D$. 
 We have that $D\setminus D_k=\{x=(x_1,\cdots, x_{n}): \text{any} \hspace{5pt} 0 < x_i \le \frac{1}{k} \hspace{5pt} \text{or}\hspace{5pt} q-\frac{1}{k}\le S(x)<q\}$.

We will show that on  
$(0, \infty) \times D_k$, for some constant $M_k$, the coefficients $\mu=-\theta(x-\frac{\alpha}{\alpha_0}),$ and 
$\sigma=\sqrt{\frac{2\theta}{\alpha_0}} \diag (\sqrt{x})\Big[S(x)I-\sqrt{x}\sqrt{x}^T \Big],$ satisfy
the Conditions~1--4 in the hypothesis of Theorem~\ref{thm:khasminskii-corollary3-1}.

We first check Condition~1.
For $x,y \in D_k$, $i,j=1,\cdots,d$ we have:
\[
\left|b(x)-b(y)\right|=
\left|-\theta\left(x-\frac{\alpha}{\alpha_0}\right)+\theta\left(y-\frac{\alpha}{\alpha_0}\right)\right|= \theta |x-y|.
\]
If $i \ne j$:
\begin{align*} 
|\sigma_{ij}(x)-\sigma_{ij}(y)| &= \sqrt{\frac{2\theta}{\alpha_0}}\Big|x_i\sqrt{x_j}-y_i\sqrt{y_j}\Big|\\ 
&\le \sqrt{\frac{2\theta}{\alpha_0}} \left( \Big|x_j\sqrt{x_i}-x_j\sqrt{y_i}\Big|+ \Big|x_i\sqrt{y_j} -y_i\sqrt{y_i}\Big| \right)\\
&\leq\sqrt{\frac{2\theta}{\alpha_0}} \left( q\frac{|x_j-y_j|}{\sqrt{x_i}+\sqrt{y_i}} +  \sqrt{q} \Big| x_i-y_i \Big| \right)
\end{align*} 

If $i=j$:
\begin{align*} 
|\sigma_{ii}(x)-\sigma_{ii}(y)| &= 
\sqrt{\frac{2\theta}{\alpha_0}} \Big|\sqrt{x_i}\left(S(x)-x_i\right)-\sqrt{y_i}\left(S(y)-y_i\right)\Big|  \\
&=\sqrt{\frac{2\theta}{\alpha_0}}\Big| \sqrt{x_i}\left( S(x)-x_i-S(y)+y_i  \right)  +
                   (\sqrt{x_i}-\sqrt{y_i})(S(y)-y_i) \Big|  \\
&\le \sqrt{\frac{2\theta}{\alpha_0}}  \left(\sqrt{q} \sum_{r\neq i}|x_r-y_r| + q\frac{| x_i-y_i |}{ \sqrt{x_i}+\sqrt{y_i}} \right) \end{align*} 
Combining these three calculations verifies that Condition~1 holds; the constant $M_k$ depends on $k$ because
we use the inequality $\sqrt{x_i}+\sqrt{y_i} > \frac{2}{\sqrt{k}}$.

For Condition~2, we observe that
\[
|b(x)|=|-\theta\left(x-\alpha/\alpha_0\right)| \leq \theta \left(|x|+|\alpha|/\alpha_0\right) \leq c(1+|x|),
\]
\[
|\sigma_{ij}(x)| = \sqrt{\frac{2\theta}{\alpha_0}} x_i\sqrt{x_j}   \le 
\sqrt{\frac{2\theta q}{\alpha_0}}  x_i \le 
\sqrt{\frac{2\theta q}{\alpha_0}}(1+|x|),
\]
if $i\neq j$, and
\[
|\sigma_{ii}(x)| = 
\sqrt{\frac{2\theta}{\alpha_0}} |\sqrt{x_i}\left(S(x)-x_i\right)| \le \sqrt{\frac{2\theta q}{\alpha_0}}  \left( S(x)+|x_i|\right) <  \sqrt{\frac{2\theta q}{\alpha_0}}\left(q+|x_i|\right).
\]
This verifies that Condition 2 holds (the constant $M$ does not even depend on $k$ in this case).

To verify Conditions~3 and~4, define the function
$$ V(x)=\sum_{i=1}^d\frac{1}{x_i} + \frac{1}{q-S(x)}.$$
For ease of computation, define $V_1(x)=\sum_{i=1}^d\frac{1}{x_i}$ and $V_2(x)=\frac{1}{q-S(x)}$, so
that $V=V_1+V_2$.

First, we have that:
\begin{eqnarray*} A(x)&=&\sigma(x)\sigma(x)^T= \frac{2\theta}{\alpha_0} \diag(\sqrt{x})\Big[S(x)I-\sqrt{x}\sqrt{x}^T \Big] \Big[S(x)I-\sqrt{x}\sqrt{x}^T \Big] \diag(\sqrt{x})\\
&=& \frac{2\theta}{\alpha_0} \Big[ S(x)\diag(\sqrt{x}) -x\sqrt{x}^T\Big]\Big[ S(x)\diag(\sqrt{x})-\sqrt{x}x^T\Big]\\
&=&  \frac{2\theta}{\alpha_0} \Big[ S^2(x)\diag(x)- S(x)xx^T-S(x)xx^T+x\sqrt{x}^T\sqrt{x}x^T  \Big]\\
&=&  \frac{2\theta}{\alpha_0} \Big[ S^2(x)\diag(x)- S(x)xx^T-S(x)xx^T+xS(x)x^T  \Big]= \frac{2\theta}{\alpha_0} \Big[ S^2(x)\diag(x)- S(x)xx^T \Big],\end{eqnarray*} 
and rewrite the above $V(x)$ as $ V(x)=V_1(x)+V_2(x)$
where $V_1(x)=\sum_{i=1}^d V_{1,i}(x) =  \sum_{i=1}^d \frac{1}{x_i}$ and $V_2(x)=\frac{1}{q-S(x)}.$ We will deal with $V_1(x)$ and $V_2(x)$ separately. 

$V_1(x)$ has the derivatives as following:
$$ \frac{\partial V_1}{\partial x_i}=-\frac{1}{x_i^2}, \hspace{30pt}  \frac{\partial^2 V_1}{\partial x_i \partial x_j}=\frac{2}{x_i^3} \cdot \delta_{i,j},$$
and applying  the operator $L$ to $V_1(x)$, we have 
\begin{eqnarray*}  
LV_1(x)&=& \sum_{i=1}^d (-\theta)(x_i-\frac{\alpha_i}{\alpha_0})(-\frac{1}{x_i^2}) +\frac{1}{2} \sum_{i=1}^d \frac{2\theta}{\alpha_0} \Big[S(x)x_i(S(x)-x_i)\Big]\frac{2}{x_i^3} \\
&=& \sum_{i=1}^d \frac{\theta}{x_i} -\theta\sum_{i=1}^d \frac{\alpha_i}{\alpha_0 x_i^2}+\frac{2\theta}{\alpha_0}S(x) \cdot  \sum_{i=1}^d\frac{S(x)-x_i}{x_i^2} \\
&=& \sum_{i=1}^d \frac{\theta}{x_i} -\theta\sum_{i=1}^d \frac{\alpha_i}{\alpha_0 x_i^2}+\frac{2\theta}{\alpha_0}S^2(x) \cdot \sum_{i=1}^d\frac{1}{x_i^2} -\frac{2\theta}{\alpha_0}S(x) \cdot \sum_{i=1}^d \frac{1}{x_i}\\
&=&  \sum_{i=1}^d \frac{\theta}{\alpha_0} \Big[\Big(\alpha_0-2S(x)\Big) \frac{1}{x_i}  + \Big(2S^2(x)-\alpha_i\Big) \frac{1}{x_i^2} \Big]
\triangleq  \sum_{i=1}^d V_{1}^L(x_i).
\end{eqnarray*}  

To check if $\frac{LV_1}{V_1}$ has a finite upper bound, it is sufficient to check if $\frac{V_1^L(x_i) }{V_{1,i}(x_i)}$ is upper bounded for each $i$:
Indeed,
\begin{eqnarray*}  
\frac{V_1^L(x_i) }{V_{1,i}(x_i)} &=& \frac{ \frac{\theta}{\alpha_0} \Big[\Big(\alpha_0-2S(x)\Big) \frac{1}{x_i}  + \Big(2S^2(x)-\alpha_i\Big) \frac{1}{x_i^2} \Big] }{1/x_i}\\
&=& \frac{\theta}{\alpha_0} \Big[\Big(\alpha_0-2S(x)\Big) + \Big(2S^2(x)-\alpha_i\Big) \frac{1}{x_i} \Big]\le\theta + \frac{\theta}{\alpha_0} \Big(2S^2(x)-\alpha_i\Big) \frac{1}{x_i} 
\le\theta.
\end{eqnarray*}  
The last step holds because $2S^2(x)-\alpha_i < 0$ on $D$. Therefore, $V_1^L(x_i) \le \theta V_{1,i}(x_i)$ which implies 
$$LV_1(x) \le \theta  V_1(x).$$

Now we check the conditions for $V_2(x)$. The derivatives are:
$$ \frac{\partial V_2}{\partial x_i}=\frac{1}{(q-S(x))^2}, \hspace{30pt}  \frac{\partial^2 V_2}{\partial x_i \partial x_j}=\frac{2}{(q-S(x))^3},$$
and applying the operator $L$ to $V_2(x)$ we have 
\begin{eqnarray*}  
LV_2(x) &=&  \sum_{i=1}^d (-\theta)(x_i-\frac{\alpha_i}{\alpha_0}) \cdot \frac{1}{(q-S(x))^2} + \frac{1}{2} \sum_{i=1}^d\sum_{j=1,j\ne i}^d  \frac{2\theta}{\alpha_0}(-S(x)x_ix_j)\cdot \frac{2}{(q-S(x))^3}\\
& &+\frac{1}{2} \sum_{i=1}^d  \frac{2\theta}{\alpha_0} S(x)x_i(S(x)-x_i) \cdot \frac{2}{(q-S(x))^3}\\
&=& \sum_{i=1}^d \frac{-\theta(x_i-\frac{\alpha_i}{\alpha_0}) }{(q-S(x))^2} + \frac{2 \theta S(x)}{\alpha_0 (q-S(x))^3} \Big\{ -\sum_{i=1}^d\sum_{j=1,j\ne i}^d x_ix_j + \sum_{i=1}^d x_i(S(x)-x_i)\Big\},
\end{eqnarray*}  
The last term can be written as
\begin{eqnarray*}  
  -\sum_{i=1}^d\sum_{j=1,j\ne i}^d x_ix_j + \sum_{i=1}^d x_i(S(x)-x_i)&=&\sum_{i=1}^d \Big\{ x_i(S(x)-x_i) - x_i \sum_{j=1,j\ne i}^d x_j \Big\}\\
=  \sum_{i=1}^d \Big\{ x_iS(x)-x_i^2 - x_i \sum_{j=1,j\ne i}^d x_j \Big\}&=&  \sum_{i=1}^d \Big\{ x_iS(x)-x_iS(x)\Big\}
= 0.
\end{eqnarray*}  
Then
\begin{eqnarray*}  
LV_2(x) = \sum_{i=1}^d \frac{-\theta(x_i-\frac{\alpha_i}{\alpha_0}) }{(q-S(x))^2} =-\theta \cdot \frac{S(x)-1}{(q-S(x))^2}.
\end{eqnarray*}  
So 
$$
\frac{LV_2(x)}{V_2(x)}=-\theta \cdot \frac{S(x)-1}{q-S(x)}  ,
$$
which is less or equal to 0 if $S(x) \ge 1$. If $S(x) <1$,
$$
\frac{LV_2(x)}{V_2(x)}=-\theta \cdot \frac{S(x)-1}{q-S(x)}   \le -\theta \cdot \frac{S(x)-1}{1-S(x)}=\theta.
$$
Therefore, $$LV_2(x) \le \theta  V_2(x).$$

Thus,
$$LV(x)=LV_1(x)+LV_2(x) \le \theta \Big(V_1(x)+V_2(x)\Big)=\theta V(x).$$
Hence, Condition 3  is verified by $V$.

Now, for  Condition 4, we check that  
$$\inf_{t>0,x \in D\setminus D_k} V(x) = \inf_{t>0,x \in D\setminus D_k} \Big \{\sum_{i=1}^d \frac{1}{x_i}+\frac{1}{q-S(x)} \Big\}
 \to \infty \hspace{20pt} as \hspace{5pt} k \to \infty, $$
on
$D\setminus D_k=\{x=(x_1,\cdots, x_{n}): \text{any} \hspace{5pt} 0 < x_i \le \frac{1}{k} \hspace{5pt} \text{or}\hspace{5pt} q-\frac{1}{k}\le S(x)<q\}$. 
Indeed, when $k \to \infty$, either at least one of the $\frac{1}{x_i} \to \infty$ or  $\frac{1}{q-S(x)} \to \infty$, so that  $\inf_{t>0,x \in D\setminus D_k} V(x) \to \infty$.
\end{proof}

\section{The invariant distribution}
Here we show that the DD process indeed has the prescribed Dirichlet distribution as its steady distribution.

\begin{theorem} Let $\alpha_1,\dots,\alpha_d>2$. Then Equation (\ref{canonicalSDE}) with initial condition
$X_1(0),\cdots,X_d(0) > 0$ and $ \sum_{i=1}^d X_i(0)=1,$ has a Dirichlet invariant measure, with parameter $\alpha$.
\end{theorem}
\begin{proof}
By Theorem~\ref{thm:existenceweaksolutions}, it is enough to prove that  the solution to Equation~\eqref{truncatedSDE} has Dirichlet invariant distribution with parameter $\alpha$. 

Equation~\eqref{truncatedSDE} has the following coefficient functions: 

\begin{eqnarray*} 
\mu(x) &=& -\theta(x-\frac{\alpha}{\alpha_0}),\\
\sigma(x)&=& \sqrt{\frac{2\theta}{\alpha_0} }\cdot \begin{bmatrix}
    \sqrt{x_1}(1-x_1) & -x_1 \sqrt{x_2} & \dots  & -x_1\sqrt{x_{d-1}} & -x_1\sqrt{1-\sum_{i=1}^{d-1}x_i} \\
     -x_2 \sqrt{x_1} & \sqrt{x_2}(1-x_2) & \dots  & -x_2 \sqrt{x_{d-1}}   & -x_2\sqrt{1-\sum_{i=1}^{d-1}x_i} \\
    \vdots & \vdots &  \ddots & \vdots  & \vdots\\
   -x_{d-1} \sqrt{x_1} &  -x_{d-1}\sqrt{x_2} & \dots  & \sqrt{x_{d-1}}(1-x_{d-1}) &  -x_{d-1}\sqrt{1-\sum_{i=1}^{d-1}x_i}
\end{bmatrix}.\\
a(x)&=&\sigma(x)\sigma^T(x) =\frac{2\theta}{\alpha_0} \cdot
\begin{bmatrix}
    x_1(1-x_1) & -x_1 x_2 & \dots  & -x_1x_{d-1} \\
     -x_2 x_1 & x_2(1-x_2) & \dots  & -x_2 x_{d-1}   \\
    \vdots & \vdots &  \ddots & \vdots \\
   -x_{d-1}x_1 &  -x_{d-1}x_2 & \dots  & x_{d-1}(1-x_{d-1})
\end{bmatrix}.
\end{eqnarray*}

We make use of the following theorem:
\begin{theorem}[Fokker--Planck/Forward Kolmogorov equation for the steady distribution; see, e.g.,  \cite{S94,S88}] \label{FPE}
Assume the coefficients $\mu$ and $\sigma$ are independent of time $t$, smooth enough to make desired
derivatives exist and continuous, and the invariant measure has a density, then the steady state probability density
function $f(x)$  verifies the  following partial differential equation: 
\begin{equation} \label{fp-equation}
\sum_{i=1}^d \frac{\partial}{\partial x_i} [\mu_i(x)f(x)]=\frac{1}{2} \sum_{i=1}^d \sum_{j=1}^d \frac{\partial^2}{\partial x_i \partial x_j} [a_{ij}(x)f(x)],
\end{equation}
where $a=\sigma \sigma^T.$

Conversely, if a function $f$ is a solution of~\eqref{fp-equation}
satisfying the appropriate positivity and normalization conditions, then $f$ is the steady state density function. 
\end{theorem}

Consider the function 
$g(x_1,\dots,x_{d-1},\alpha)=\left[\prod_{i=1}^{d-1} x_i^{\alpha_i-1}\right] \cdot x_d^{\alpha_d-1}$,
where $x_d$ denotes $1-\sum_{i=1}^{d-1}x_i$. This is an unnormalized version of the Dirichlet density with
parameter $\alpha$ in Equation~\eqref{eq:DirichletDensity}. We just have to verify that $g$ satisfies
Equation~\eqref{fp-equation}, but in $d-1$ dimensions. That is, we want to prove:
$$\sum_{i=1}^{d-1} \frac{\partial}{\partial x_i} [\mu_i(x)g(x)]=\frac{1}{2} \sum_{i=1}^{d-1} \sum_{j=1}^{d-1} \frac{\partial^2}{\partial x_i \partial x_j} [a_{ij}(x)g(x)].$$

So after substituting coefficient functions $\mu$ and $\sigma$ and canceling constant $\frac{\theta}{\alpha_0}$, we claim:
\begin{equation} \label{F-Pclaim}
-\sum_{i=1}^{d-1} \frac{\partial}{\partial x_i} [(\alpha_0x_i-\alpha_i)g]
=\sum_{i=1}^{d-1} \frac{\partial^2}{\partial x_i^2} [x_i(1-x_i)g]
-\sum_{i=1}^{d-1}\sum_{j=1,j \ne i}^{d-1}  \frac{\partial^2}{\partial x_i \partial x_j} [x_i x_j g].
\end{equation} 

We first compute the derivatives of $g$. For $i,j\in\{1,\dots,d-1\}$, denote
$g_i^{'}=\frac{\partial g }{\partial x_i}$, 
$g_{ii}^{''}=\frac{\partial^2 g }{\partial x_i^2}$, and $g_{ij}^{''}=\frac{\partial^2 g }{\partial x_i \partial x_j}$. Then:
\begin{eqnarray*} 
g_i^{'} &=& 
\left[
\prod_{i\neq i,d}x_i^{\alpha_i-1}
\right]
\left[
(\alpha_i-1)x_i^{\alpha_i-2}x_d^{\alpha_d-1}+(-1)(\alpha_d-1)x_i^{\alpha_i-1}x_d^{\alpha_d-2}
\right]\\
&=& g \cdot \Big[ \frac{\alpha_i-1}{x_i} -\frac{\alpha_d-1}{x_d}\Big],
\end{eqnarray*}
\begin{eqnarray*} 
g_{ii}^{''} &=& \frac{\partial}{\partial x_i} \Big[ g\cdot \frac{\alpha_i-1}{x_i} -g \cdot \frac{\alpha_d-1}{x_d}\Big]\\
&=& g \cdot \Big[\frac{(\alpha_i-1)(\alpha_i-2)}{x_i^2}-\frac{2(\alpha_i-1)(\alpha_d-1)}{x_ix_d} +\frac{(\alpha_d-1)(\alpha_d-2)}{x_d^2}    \Big],
\end{eqnarray*}
\begin{eqnarray*} 
g_{ij}^{''} &=& \frac{\partial}{\partial x_j} \Big[ g\cdot \frac{\alpha_i-1}{x_i} -g \cdot \frac{\alpha_d-1}{x_d}\Big]\\
&=& g \cdot \Big[\frac{(\alpha_i-1)(\alpha_j-1)}{x_ix_j}-\frac{(\alpha_i-1)(\alpha_d-1)}{x_ix_d}-\frac{(\alpha_j-1)(\alpha_d-1)}{x_jx_d} +\frac{(\alpha_d-1)(\alpha_d-2)}{x_d^2}\Big].
\end{eqnarray*}

Next we compute each term inside the summation of Equation (\ref{F-Pclaim}).
\begin{eqnarray*} 
 \frac{\partial [(\alpha_0x_i-\alpha_i)g]}{\partial x_i} &=& g \cdot \Big[\alpha_0\alpha_i  -\frac{\alpha_i(\alpha_i-1)}{x_i} -\frac{\alpha_0(\alpha_d-1)x_i -\alpha_i(\alpha_d-1)}{x_d}  \Big],
\end{eqnarray*}
\begin{eqnarray*} 
\frac{\partial^2  [x_i(1-x_i)g]}{\partial x_i^2}
&=& g \cdot \Big[-\alpha_i(\alpha_i+1)+ \frac{\alpha_i(\alpha_i-1)}{x_i} + \frac{2(\alpha_d-1)(\alpha_i+1)x_i -2\alpha_i(\alpha_d-1)}{x_d}  \\
& &+\frac{x_i(1-x_i)(\alpha_d-1)(\alpha_d-2)}{x_d^2} \Big],
\end{eqnarray*}

\begin{eqnarray*} 
 \frac{\partial^2  [x_i x_j g]}{\partial x_i \partial x_j} 
 &=&   g \cdot \Big[\alpha_i\alpha_j-\frac{\alpha_j x_i(\alpha_d-1)+\alpha_ix_j(\alpha_d-1)}{x_d}+\frac{(\alpha_d-1)(\alpha_d-2)x_ix_j}{x_d^2} \Big].
\end{eqnarray*}

Then 
\begin{eqnarray*} 
\sum_{j=1, j \ne i}^{d-1} \frac{\partial^2  [x_i x_j g]}{\partial x_i \partial x_j} &=& g \cdot \Big[\alpha_i (\alpha_0-\alpha_i-\alpha_d) -\frac{(\alpha_d-1) (\alpha_0-\alpha_i-\alpha_d)x_i}{x_d}\\
&&- \frac{\alpha_i(\alpha_d-1) (1-x_i-x_d)}{x_d}  +\frac{(\alpha_d-1)(\alpha_d-2)x_i(1-x_i-x_d)}{x_d^2} \Big].
\end{eqnarray*}

Therefore, the left hand side (LHS) of  the summation of Equation (\ref{F-Pclaim}) is:
\begin{eqnarray*} 
\text{LHS}=-\frac{\partial}{\partial x_i}\Big[(\alpha_0x_i-\alpha_i)g\Big]= g \cdot \Big[-\alpha_0\alpha_i +\frac{\alpha_i(\alpha_i-1)}{x_i} +\frac{\alpha_0(\alpha_d-1)x_i -\alpha_i(\alpha_d-1)}{x_d}  \Big].
\end{eqnarray*}
while the right hand side (RHS) is:
\begin{eqnarray*} 
\text{RHS} &=& \frac{\partial^2}{\partial x_i^2}\Big[x_i(1-x_i)g \Big] - \sum_{j=1,j \ne i}^{d-1} \frac{\partial^2}{\partial x_i \partial x_j}\Big[ x_i x_j g\Big]\\
&=& g \cdot \Big[-\alpha_0\alpha_i +\frac{\alpha_i(\alpha_i-1)}{x_i}  -\frac{\alpha_i(\alpha_d-1)}{x_d} +\frac{\alpha_0(\alpha_d-1)x_i}{x_d}\Big].
\end{eqnarray*}
Hence LHS=RHS, which implies $g(x,\alpha)$, or equivalently $f(x,\alpha)$, verifies   Equation~\eqref{fp-equation}. 
\end{proof}

\section{The Aggregation Property for the DD process}\label{sec:aggregation}
The DD process satisfies the a version of the aggregation property (see Section~\ref{sec:dirichletDistribution}),
in the following sense:
\begin{theorem}
Fix $\theta>0$ and $\alpha=(\alpha_1,\dots,\alpha_d)^T$, with $\alpha_1,\dots,\alpha_d>0$, and let
$\left\{X(t):t\geq0\right\}$ be a solution for SDE~(\ref{canonicalSDE}).
Let $\{B_1,B_2,\dots,B_k\}$ be a partition of the set $\{1,2,\dots,d\}$, and define the process $\{Y(t):t\geq0\}$
by $Y_i(t) = \sum_{j\in B_i}X_j(t)$ for $j=1,\dots,k$. Then 
$\left\{Y(t):t\geq0\right\}$ is a solution for SDE~(\ref{canonicalSDE})
in $k$ dimensions with parameters $\theta$ and 
$\alpha'=\left(\sum_{j\in B_1} \alpha_j,\sum_{j\in B_2} \alpha_j,\dots,\sum_{j\in B_k} \alpha_j\right)$.

\end{theorem}
\begin{proof}
Assume that
$\left((\Omega,\F,P),\{\F_t\},W(t),X(t)\right)$
is a solution of SDE~\eqref{canonicalSDE}
and define $\{\tilde{W}(t)\}$ as a process in $\mathbb{R}^k$ with
components given by 
$$\tilde{W}_i(t) = \sum_{j\in B_i} \int_0^t\frac{ \sqrt{X_j(t)} }{\sqrt{\sum_{r\in B_i} X_r(t)}}\,dW_j(t), \quad i=1,\dots,k.$$
Since $\{X(t):t\geq0\}$ is a semimartingale adapted to $\{\F_t\}$, 
then $\{\tilde{W}(t):t\geq0\}=\{(\tilde{W}_1(t),\dots,\tilde{W}_k(t))^T:t\geq0\}$ is
a continuous local martingale. Using L\'evy's characterization of Brownian motion (see, e.g., \cite{KS91},
Theorem 3.16, pg.~157), we will prove now that $\{\tilde{W}(t):t\geq0\}$ is a standard Brownian
motion in $\R^k$ by checking that the quadratic cross-variation of $Y(t)$ equals $tI_k$.

Indeed, 
\begin{align*}
\left\langle \tilde{W}_i(t),\tilde{W}_i(t) \right\rangle_t & =    
\sum_{j\in B_i}\sum_{j'\in B_i}
\int_0^t
\left(
\frac{ \sqrt{X_j(t)} }{\sqrt{\sum_{r\in B_i} X_r(t)}}
\right)
\left(
\frac{ \sqrt{X_{j'}(t)} }{\sqrt{\sum_{r\in B_i} X_r(t)}}
\right)
\, d\left\langle W_j,W_{j'}  \right\rangle
\\
& =
\int_0^t  \sum_{j\in B_i} \left(  
       \frac{ \sqrt{X_j(t)} }{\sqrt{\sum_{r\in B_i} X_r(t)}}
   \right)^2\,dt  \qquad \textrm{(since $d\left\langle W_j,W_{j'}  \right\rangle=\delta_{jj'}\,dt$)}
   \\
& = \int_0^t \frac {\sum_{j\in B_i} X_j(t)}{\sum_{r\in B_i} X_r(t)} \,dt \\
& = t,
\end{align*}
for $i=1,\dots,k$. Also, if $i\neq i'$,
\begin{align*}
\left\langle \tilde{W}_i(t),\tilde{W}_{i'}(t) \right\rangle_t  & = 
\sum_{j\in B_i}\sum_{j'\in B_{i'}}
\int_0^t
\left(
\frac{ \sqrt{X_j(t)} }{\sqrt{\sum_{r\in B_i} X_r(t)}}
\right)
\left(
\frac{ \sqrt{X_{j'}(t)} }{\sqrt{\sum_{r\in B_i} X_r(t)}}
\right)
\, d\left\langle W_j,W_{j'}  \right\rangle
\\
& = 0,
\end{align*}
because $B_i\cap B_{i'}=\emptyset$, so $j$ and $j'$ can never be the same and therefore
$d\left\langle W_j,W_{j'}  \right\rangle=0$ always.

Now let $A$ be a $k\times d$ matrix defined by $a_{ij}=1$ if $j\in B_i$, and 0 otherwise. Then $AX(t) = Y(t)$,
$AdX(t) = dY(t)$, and
$A\alpha = \alpha'$. We will multiply both sides of SDE~(\ref{canonicalSDE}) by $A$ on the left and verify
that the right hand side has the same form as the original SDE.

First  observe that
\begin{equation}\label{chunk1}
A\left(-\theta\left( {X}- \frac{\alpha}{\alpha_0}\right)dt\right)=-\theta\left( {AX}- \frac{A\alpha}{\alpha_0}\right)dt =
-\theta\left( {Y}- \frac{\alpha'}{\alpha_0}\right)dt
\end{equation}
and  that $\alpha_0$ is the sum of the components of $\alpha'$.
Then notice that
\begin{align}
A\diag\left(\sqrt{X}\right)I_ddW &=
\left(\sum_{j\in B_1}\sqrt{X^j}dW^j,\dots,\sum_{j\in B_k}\sqrt{X^j}dW^j \right)^T \notag\\
&=\left(\sqrt{\sum_{j\in B_1}X^j}d\tilde{W}^1,\dots,\sqrt{\sum_{j\in B_k}X^j}d\tilde{W}^k\right)^T\notag\\
&=\diag\left(\sqrt{Y}\right)I_kd\tilde{W}\label{chunk2}
\end{align}
and that
\begin{align}
A\diag\left(\sqrt{X}\right)\sqrt{X} \sqrt{X}^TdW&=AX\sqrt{X}^TdW=Y\sqrt{X}^TdW=Y\sum_j\sqrt{X^j}dW^j\notag\\
&=Y\sum_i\sum_{j\in B_i}\sqrt{X^j}dW^j=Y\sqrt{Y}^T\,d\tilde{W}=\diag\left(\sqrt{Y}\right)\sqrt{Y} \sqrt{Y}^T\,d\tilde{W}\label{chunk3}
\end{align}
Adding the final expressions of ~(\ref{chunk1}), (\ref{chunk2}), and~(\ref{chunk3}), after multiplying the last 
two by $\sqrt{\frac{2\theta}{\alpha_0}}$, we obtain
an equation of the same form as~(\ref{canonicalSDE}).
\end{proof}

\section{Simulations}

\subsection{Numerical methods}

In this section we perform simulations of the DD process by computing numerical solutions of the SDE.
For simplicity, we use the Euler--Maruyama method (see, e.g., \cite{kloeden2011numerical}), which is
a recursive first order time discretization.
The method uses the following
ingredients: an initial value $X_0\in\s$, a fixed time step $\Delta t>0$, and Brownian motion increments
$\Delta W_0,\dots,\Delta W_{N-1}$ i.i.d. $\mathcal{N}(0,\sqrt{\Delta t})$ to generate values $X(t)=X(i\Delta t)=X_i$,
$i=1,\dots,N$, that is,
for $N+1$ equally spaced values of $t$ in $[0,N\Delta t]$, according to the following iterative scheme.

The explicit scheme is
\begin{equation}\label{eq:explicitscheme}
X_{n+1} = X_n -\theta(X_n-\mu)\Delta t + \sigma(X_n)\Delta W_n
\end{equation}
The implicit version of the scheme evaluates the drift term at $X_{n+1}$ rather than $X_n$, while still
evaluating the diffusion term at $X_n$.
We build a  family of partially implicit schemes, indexed by a parameter $0\leq\rho\leq1$:
\begin{equation*}
X_{n+1} = X_n + \rho\left[ -\theta(X_{n+1}-\mu)\Delta t\right]+
                    (1-\rho)\left[-\theta(X_{n}-\mu)\Delta t\right]+ \sigma(X_n)\Delta W_n
\end{equation*}
Solving for $X_{n+1}$ yields the following iteration:
\begin{equation}
X_{n+1} = \frac{1}{1+\theta\rho\Delta t}\left[(X_n(1-\theta(1-\rho)\Delta t)+\theta\mu\Delta t + \sigma(X_n)\Delta W_n\right]
\end{equation}
Setting $\rho=0$ corresponds to the explicit scheme \eqref{eq:explicitscheme}, while $\rho=1$ results
in the fully implicit version of the scheme. Choosing $0<\rho<1$ results in an intermediate combination of
both approaches.

If $X(0)\in\bar\s$, each iteration produces a vector with entries that add up to 1 (up to numerical error).
However, both schemes can exit the simplex $\bar\s$, no matter how small is $\Delta t$, by producing a vector
$X_n$ with one or more negative entries. Therefore, we include a ``reflecting'' step in each iteration, in which
each entry of $X_n$ is replaced with its absolute value; this necessitates the renormalization of $X_n$ so
that its components add up to 1.

As can be expected from our theoretical results, when $\min(\alpha_1,\dots,\alpha_d)>2$ the simulation
rarely crosses the boundaries of $\bar\s$. However, if one or more components of $\alpha$ are small,
the frequency of ``boundary hits'' increases; this can have effects on the performance of the simulation,
as will be discussed in Section~\ref{sec:convergence}.

The algorithm was implemented in the statistical computation language and environment
\emph{R}~\cite{Rmanual2024} in the form of a package, titled ``diridiff," which is available at
 the following URL: \url{https://github.com/odelacruzc/diridiff}.
The Brownian motion increments are obtained using \emph{R}'s default pseudo-random number generation
methods for the normal distribution.

\begin{figure}
\includegraphics[scale=0.75]{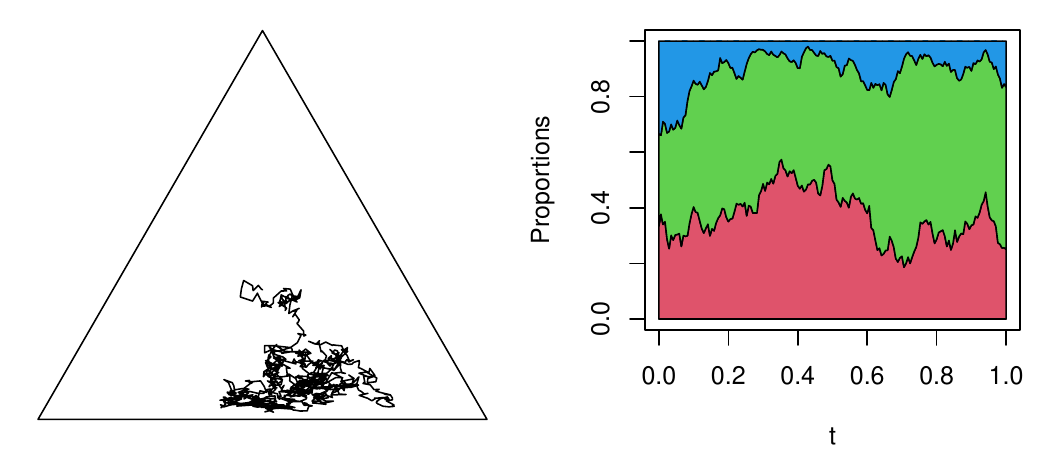}
\caption{A simulation run of the DD process in 3 components with parameters $\theta=1$ and
$\alpha=(1,1,1)$; the time range is $[0,1]$ and the starting point is $\mu=(1/3,1/3,1/3)$.
The explicit Euler--Maruyama scheme was
used with $\Delta t=10^{-3}$. The left panel presents the process in barycentric coordinates, while the
right panel shows the same process in a ``flag plot,'' with time on the $x$-axis and the components
stacked on the $y$-axis.
}\label{fig:TriangleAndFlagPlots}
\end{figure}

\subsection{The steady distribution}

Simulations show that the steady distribution of the DD process is indeed the appropriate
Dirichlet distribution. Figure~\ref{fig:fourTrianglePlots} presents four examples; each panel
contains the endpoints of $1,000$ independent runs of the DD process in 3 components,
plotted using barycentric coordinates.

\begin{figure}
\includegraphics[scale=0.4]{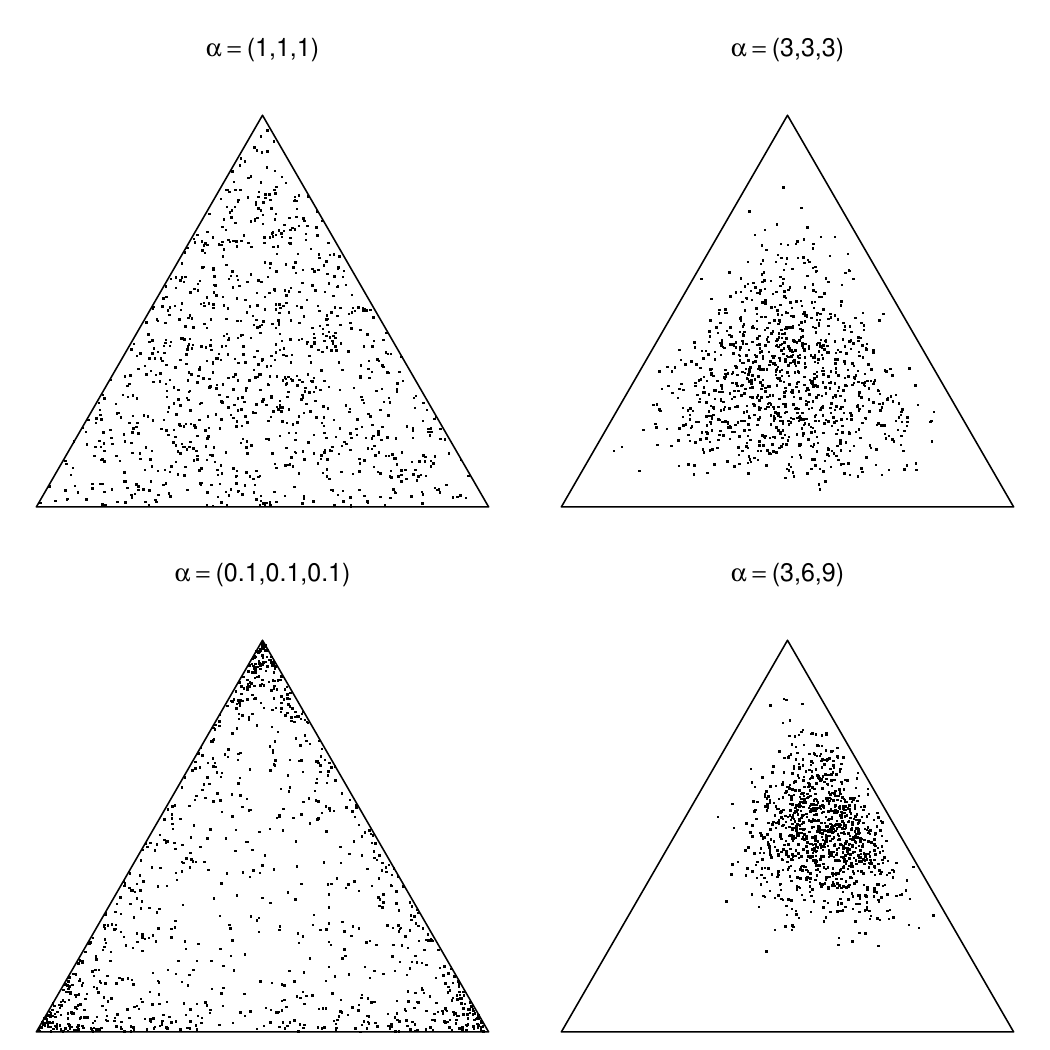}
\includegraphics[scale=0.4]{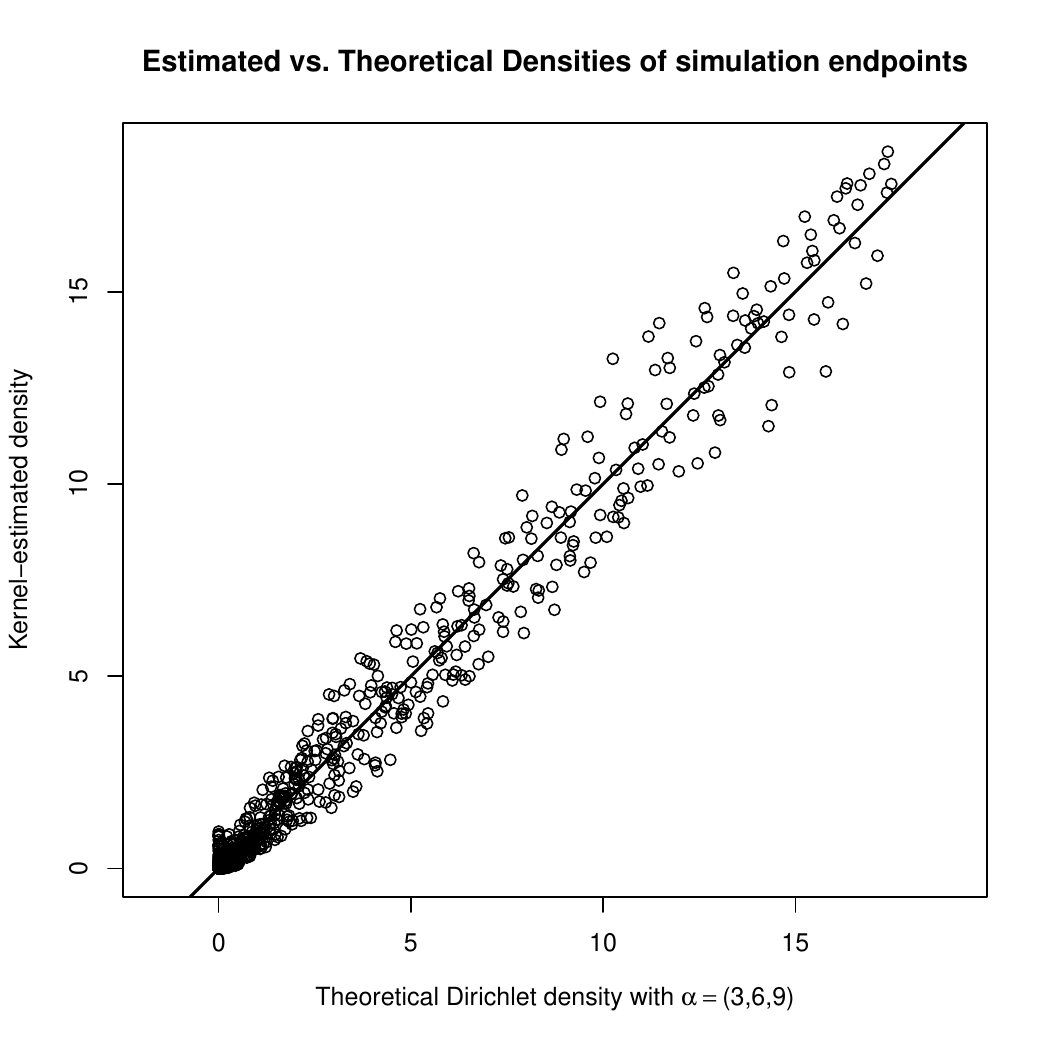}
\caption{Compositional time series data.
The plots shar in the plots.
}\label{fig:fourTrianglePlots}
\end{figure}

In low dimensions we can estimate the density for the simulation endpoints using, for example,
kernel methods. The panel on the right hand side of Figure~\ref{fig:fourTrianglePlots} shows a
comparison of the estimated density of the endpoints (computed using the algorithm \emph{kde2d}
proposed in \cite{VR-MASS2002} for 2-dimensional kernel density estimation) with the
corresponding density for the Dirichlet distribution with parameter $\alpha=(3,6,9)$.

It is difficult  in general to formally verify the fit of a multivariate distribution in high dimensions. In our
case, we can take advantage of the aggregation property in order to test multiple 1-dimensional
distributions. Figure~\ref{fig:nineQQandDensityPlots} displays the results of the following numerical
experiment: We generated 500 simulation runs of the DD process over the time interval $[0,10]$ with
parameters $\theta=1$ and  $\alpha=(1/2,1/2,\dots,1/2)\in\R^{10}$, using $\Delta t=10^{-4}$.
To investigate whether the resulting 500 endpoints in $\R^{10}$ follow the $\mathrm{Dirichlet}(\alpha)$
distribution, we add the first $k$ components; the resulting set of
numbers should follow the Beta distribution with parameters $k\cdot 1/2$ and $(10-k)\cdot 1/2$.
Figure~\ref{fig:nineQQandDensityPlots} contains the result of such comparisons, for $k=1,\dots,9$,
using QQ-plots and density plots (with empirical densities computed using the default \emph{R}
method).
Furthermore, Kolmogorov--Smirnov tests for each comparison result in very large $p$-values,
providing evidence that these univariate distributions match.  None of this is proof that the multivariate
distribution is correct, but it is highly suggestive.

\begin{figure}
\includegraphics[scale=0.4]{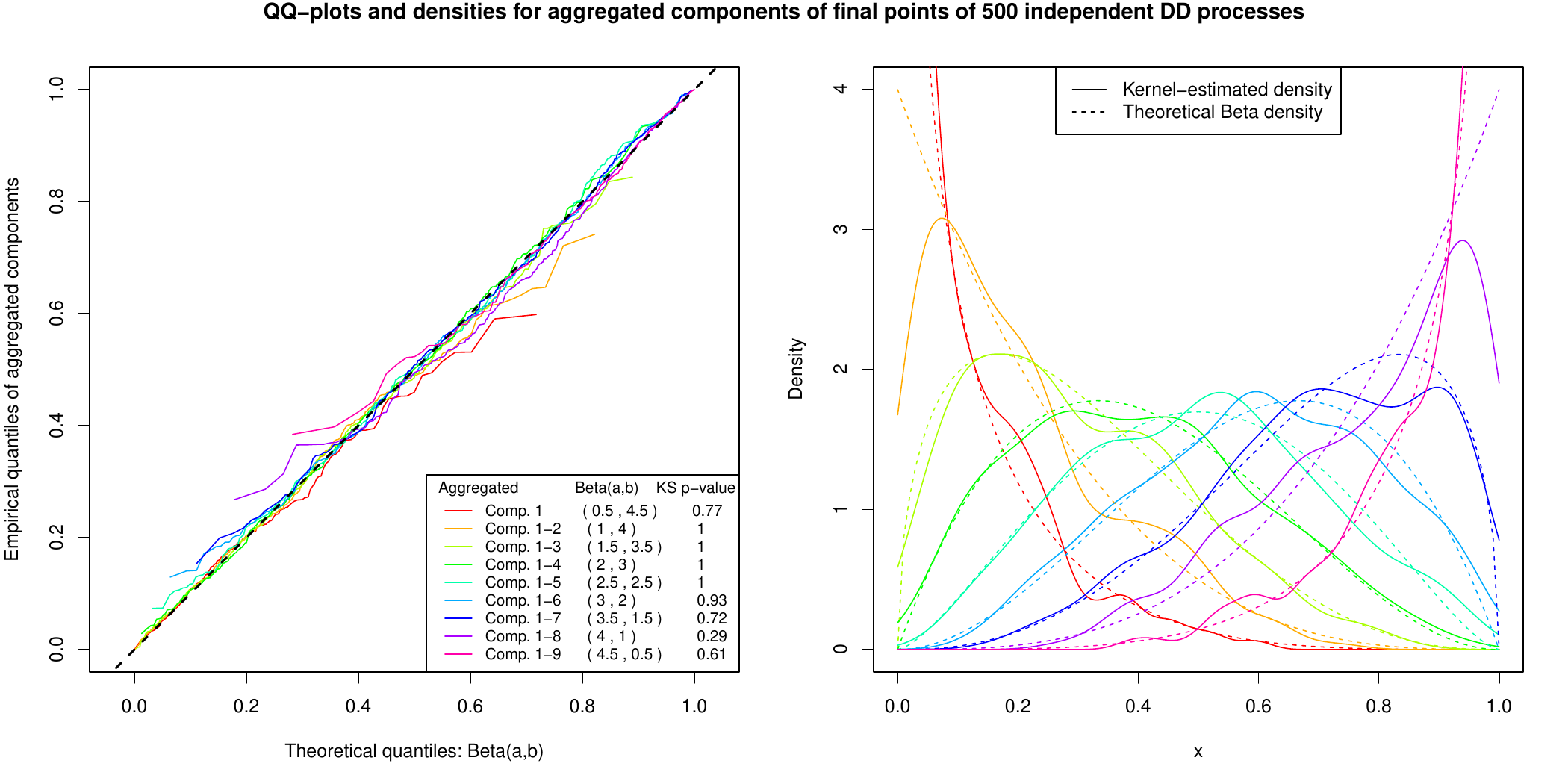}
\caption{
Results for 500 simulation runs of the DD process over the time interval $[0,10]$ with
parameters $\theta=1$ and  $\alpha=(1/2,1/2,\dots,1/2)\in\R^{10}$, using $\Delta t=10^{-4}$,
of which we keep only the endpoints.
For $k=1,\dots,9$, we add the first $k$ components, and we compare the distribution of the
resulting 500 numbers to the Beta distribution with parameters $k\cdot 1/2$ and $(10-k)\cdot 1/2$.
The panel on the left compares the quantiles (QQ-plot); the legend contains also the $p$-value
for the corresponding Kolmogorov--Smirnov (KS) test. The panel on the right displays estimated
density functions (computed using the default \emph{R} method) and the theoretical densities.
}\label{fig:nineQQandDensityPlots}
\end{figure}

\subsection{Convergence of solutions}\label{sec:convergence}

Consider a DD process $\left((\Omega,\F,P),\{\F_t\},W(t),X(t)\right)$. Denote by $X_{\Delta t}(t)$
the Euler--Maruyama approximation of $X(t)$ obtained using time step $\Delta t$ and using Brownian
increments $\Delta W_k = W(k\Delta x)-W((k-1)\Delta x)$ (interpolating linearly), over a time
interval $[0,T]$.
Theoretical results about the convergence of first order numerical schemes like Euler--Maruyama
(see, e.g., \cite{kloeden2011numerical}) state that
\begin{equation}\label{eq:convergenceOrder}
E\left(  \| X(T) - X_{\Delta t}(T) \| \right) \leq C(\Delta t)^\gamma,
\end{equation}
with order of convergence $\gamma = 1/2$,
\emph{under the assumption that the SDE's coefficients are globally Lipschitz}. As we have seen above,
that assumption does not hold in our case. However, it is not unreasonable to expect that when
$\min(\alpha_1,\dots,\alpha_d)>2$ (which guarantees the existence of strong solutions, see
Section~\ref{sec:existenceStrong}) we might achieve the same order of convergence. This seems to be
the case, even for $\alpha$ having components well below 2.

To study the convergence of the numerical algorithm as $\Delta t\to 0$, we first generate a simulation run
of $d$-dimensional Brownian motion $W$ over the interval $[0,1]$ with $\Delta t = 10^{-6}$, and use
it to generate a Euler--Maruyama solution $X$, which we regard as a proxy for ``ground truth."
Then, coarser Euler--Maruyama approximations are computed, using larger values of $\Delta t$:
$10^{-1}$, $10^{-2}$, $10^{-3}$, $10^{-4}$, and $10^{-5}$, extracting the Brownian increments from the
pre-computed $W$. This process is repeated 100 times and the results are averaged in order to
approximate the expectation in Equation~\eqref{eq:convergenceOrder}.

We performed the experiment above for the following values of $\alpha$: $(3,3,3)$, $(2,2,2)$, $(1,1,1)$,
$(0.5,0.5,0.5)$, $(0.1,0.1,0.1)$, and $(0.05,0.05,0.05)$. Fitting a least-squares line in log-log coordinates
provides the estimated convergence rate and order; the results are displayed in Figure~\ref{fig:convergenceRates}.
An approximate convergence order $\gamma=1/2$ seems to hold even for values of $\alpha$ below 1
(although the constant $C$ increases). However, the order of convergence does deteriorate as the 
components of $\alpha$ get close to zero. 

\begin{figure}
\includegraphics[scale=0.75]{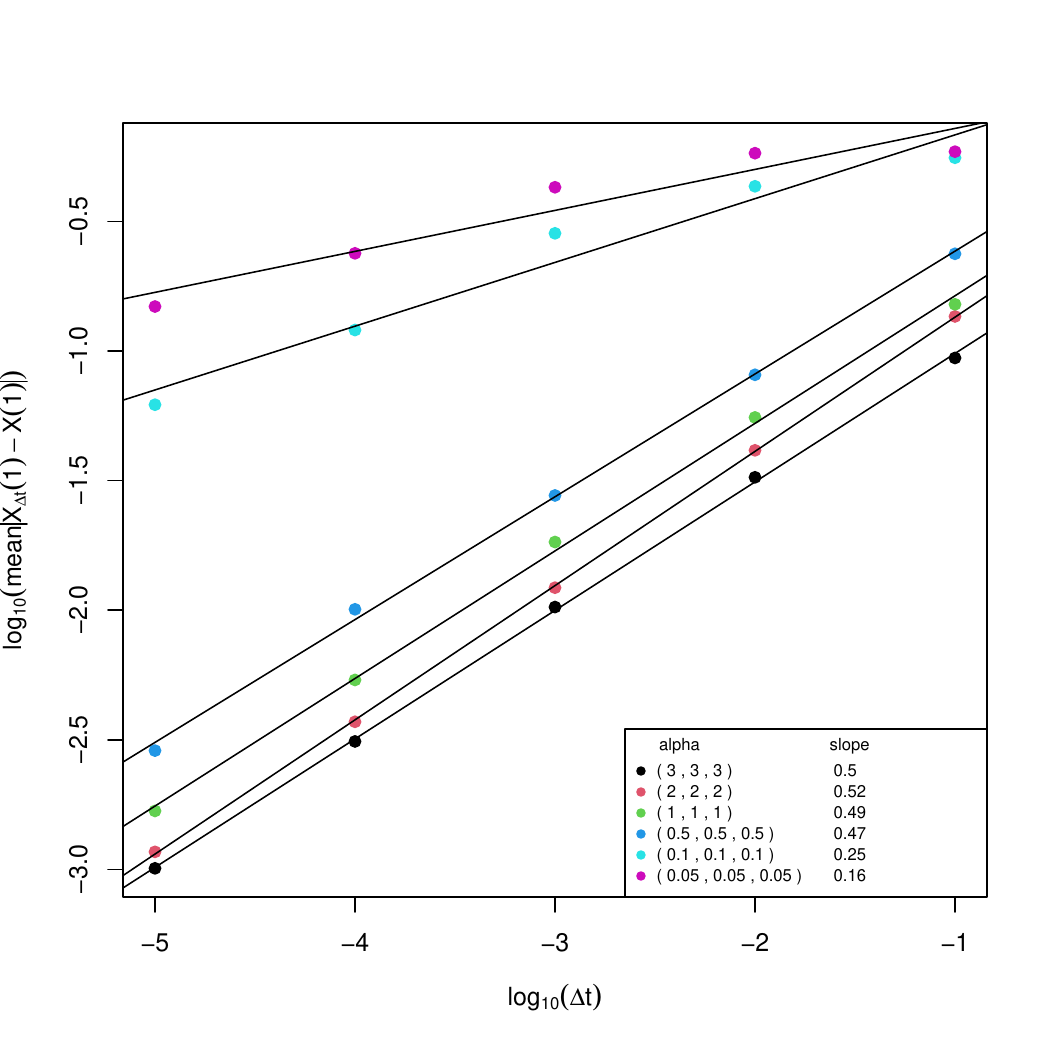}
\caption{
Numerical simulation of $E\left(  \| X(T) - X_{\Delta t}(T) \| \right)$ (see Equation~\eqref{eq:convergenceOrder}).
Each color represents
one experiment in 3 dimensions with parameter $\alpha$. A reference solution $X(t)$ for $t\in[0,1]$ was generated
using $\Delta t=10^{-6}$, and approximations were computed using larger values of $\Delta t$. Each
experiment was repeated 100 times and averaged to approximate $E\left(  \| X(1) - X_{\Delta t}(1) \| \right)$.
The lines were fitted in log scale, so the slopes are estimates of the convergence order $\gamma$, which
ends up being close to the theoretical value of $\gamma=1/2$ for all but the smallest values of $\alpha$.
}\label{fig:convergenceRates}
\end{figure}

Future work will explore the use of higher order numerical methods for solving SDEs, as well
as the possibility of having $\Delta t$ vary as the process approaches the boundaries, as this would likely
reduce the frequency of boundary hits and thus improve performance.

\section{Conclusions}

We have introduced a stochastic process which we believe is useful for modeling the evolution in time
of compositional measurements (i.e., a vector of non-negative values that add up to a total of 1). This model
is a diffusion, as it is defined as a solution for a SDE in the It\^o sense, and it has a Dirichlet distribution as
its steady distribution. We have named this process \emph{Dirichlet Diffusion} (DD).

The process is confined to the closed simplex $\bar\s\subset\R^d$, which has dimension $d-1$; also, the 
coefficients of the SDE~\eqref{canonicalSDE}
are locally Lipschitz, but not globally Lipschitz. Consequently, the usual theorems do not apply directly, and
establishing the existence of solutions required a somewhat delicate analysis. We were able to establish
the existence of strong solutions under the assumption $\min(\alpha_1,\dots,\alpha_d)>2$; for the general
case we were only able to establish the existence of weak solutions, but also that these solutions remain
confined to the closed simplex $\bar\s$ without need for reflecting boundaries. The question of whether
strong solutions exist when $\min(\alpha_1,\dots,\alpha_d)\leq2$ remains open; it is likely that an answer
will require a different approach.

An interesting property of DD is that $\mu$, which can be regarded as a shape parameter for the Dirichlet
distribution, appears only in the drift term, while $\alpha_0$, which can be regarded as a concentration or
spread parameter, appears only the diffusion term; the overall structure of the SDE remains unchanged
(for example, when components of $\alpha$ are smaller than 1 the Dirichlet density becomes unbounded,
while the terms in the SDE remain bounded).
This suggests that generalizations of the process can be obtained by considering models in which
$\mu$ or $\alpha_0$ are allowed to change in time.

A useful feature of DD as a model for compositional data is the property of aggregation 
(Section~\ref{sec:aggregation}). 
The example discussed in the Introduction is typical: A compositional phenomenon can often be described
at different levels of detail, like the nested taxonomic ranks of living organisms. This means that important features
of the process remain invariant regardless of the level of detail chosen, since $\theta$ and $\alpha_0$ 
remain the same for different levels
of aggregation. In particular, this indicates that estimation of those parameters can be carried out at the level of
aggregation that allows for more precision.

Work in progress includes methods for estimation and goodness of fit for DD processes.
The DD process can be described as a null model, describing the evolution of a system due purely to ``noise." 
Generalizations of the DD process can be obtained by allowing time-varying $\mu$ or $\alpha_0$, as mentioned
above, or by including a covariance matrix affecting the diffusion term, creating correlations at the time microscale.
In this case, the DD process can be used as the null hypothesis against which to test for the presence of such
further structure.

\appendix

\section{Proof of Theorem~\ref{thm:existenceweaksolutions}}

\begin{proof}[Proof of Theorem~\ref{thm:existenceweaksolutions}]

First we will show that weak solutions exist for modified versions of the SDEs under consideration.

We will use this fact: if a continuous function $f$ is defined on a compact, convex set $K\subset\R^d$, it can be
extended to a bounded continuous function $f^*$ defined on all of $\R^d$. Indeed, for all $x\in\R^d$,
define $f^*(x)=f(\pi_K(x))$, where $\pi_K$ is the projection map onto $K$, which is continuous.

Consider the SDE \eqref{canonicalSDE}; the argument for \eqref{adjustedSDE} is similar.

We constrain the domain to be the closed
cube $K=[0,1]^d$.
Each component of the drift term $b(X)=-\theta(X-\mu)$ 
and of the diffusion term 
$$\sigma(X)=\sqrt{\frac{2\theta}{\alpha_0}}\diag\left(\sqrt{X}\right)\left(S(X)I_d-\sqrt{X} \sqrt{X}^T\right)$$
is continuous and bounded on $K$; we extend each of them to continuous bounded functions 
defined on all of $\R^d$, by the method described above, to obtain
$b^*(X)$ and $\sigma^*(X)$, respectively, 
and define a modified version
of the SDE (which agrees with the original one on $K$). By the classical existence
result of Skhorohod (see, e.g., \cite{KS91}), given an initial distribution on $\R^d$ with support
contained in $\s$, there exists a unique (in distribution) weak solution for the modified SDE.

Let $\left((\Omega,\F,P),\{\F_t\},W(t),X(t)\right)$ be such a solution.
We will show now that $X(t)$ stays within $\bar\s\subset K$ for all $t\geq0$, $P$-a.s., and therefore
the solution is also a solution for the original SDE \eqref{canonicalSDE}.

First, we show that none of the components $X_i$, $i=1,\dots,d$, go below 0. 
To find the equation that determines
the behavior of $X_i$, we multiply the modified SDE on the left by the vector $e_i^T=[0,\dots,1,\dots,0]$
(with 1 in the $i$-th component) and we obtain, by a trivial application of Ito's formula, the equation
\begin{equation}\label{eq:sdeForXi}
dX_i = e_i^Tb^*(X)\,dt+e_i^T\sigma^*(X)\,dW
\end{equation}
Notice that $e_i^Tb^*(X)$ depends only on $X_i$:
$$
g(X_i):=e_i^Tb^*(X)=
\begin{cases}
\theta\mu_i & \textrm{if $X_i\leq 0$}\\
-\theta(X_i-\mu_i) & \textrm{if $0\leq X_i\leq 1$}\\
-\theta(1-\mu_i) & \textrm{if $1\leq X_i$}\\
\end{cases}
$$
Now, $e_i^T\sigma^*(X)$ is a row vector which, when $X\in K$, has the form
$$
\sqrt{X_i}\left(
\sqrt{X_iX_1},\dots,\sqrt{X_iX_{i-1}},1-X_i,\sqrt{X_i,X_{i+1}},\dots,\sqrt{X_iX_d}
\right)
$$
When $X_i\leq0$, $\pi_K(X)$ is on the face of the boundary of $K$ determined by
the hyperplane $X_i=0$, and therefore $e_i^T\sigma^*(X)=(0,\dots,0)$. Thus,
$e_i^T\sigma^*(X)$ can be written as $\sqrt{\max(0,X_i)}h(X)$, where $h(X)$ is defined
and continuous on $\R^d$. If we denote by $M$ the scalar local martingale defined by
$$
M(t)=\int_0^t h\left(X(s)\right)dW(s),
$$
then $X_i$ satisfies the martingale-driven SDE
\begin{equation}\label{eq:martingaleSdeForXi}
dX_i=g(X_i)\,dt+\sqrt{\max(0,X_i)}\,dM.
\end{equation}
That is, $\left((\Omega,\F,P),\{\F_t\},M(t),X_i(t)\right)$ is a weak solution of the
martingale-driven SDE \eqref{eq:martingaleSdeForXi}.

Now we apply the Comparison Theorem; see, for example, Proposition~2.18 in \cite{KS91},
modified to use a local martingale instead of Brownian motion.
Define the function
$$
g_0(x)=
\begin{cases}
0 & \textrm{if $x\leq \mu_i$}\\
-\theta(x-\mu_i) & \textrm{if $\mu_i\leq x\leq 1$}\\
-\theta(1-\mu_i) & \textrm{if $1\leq x$}\\
\end{cases}
$$
Clearly, $g_0$ is continuous and bounded, and $g_0(x)\leq g(x)$ for all $x\in\R$. 
The process constantly equal to zero is the solution of the SDE 
$$
dX=g_0(X)\,dt+\sqrt{\max(0,X)}\,dM.
$$
with initial value equal to zero. Since $0\leq X_i(0)$ a.s., by the Comparison Theorem,
we conclude that $0\leq X_i(t)$ for all $t\geq0$, a.s.

Second, we show that
$X(t)$ remains on the affine space
$x_1+\cdots+x_d=1$ for $0\leq t$
a.s. Indeed, let $S(t)=S(X(t))=\sum_i X_i(t)$; multiplying 
Equation~\eqref{canonicalSDE}
on the left by $1^T$ and a trivial application of Ito's Lemma, we get that $S(t)$ satisfies the equation
$$
dS(t) = -\theta(S(t)-1)dt +\sqrt{\frac{2\theta}{\alpha_0}}(1-S(t))\sqrt{X}^TdW,
$$
because
\begin{align*}
1^T\diag\left(\sqrt{X}\right)\left(I_d-\sqrt{X}\sqrt{X}^T\right) &= 
\left(\sqrt{X}^T - \sqrt{X}^T\sqrt{X}\sqrt{X}^T \right)\\
 &=(1-S)\sqrt{X}^T,
\end{align*}
since $\sqrt{X}^T\sqrt{X} = S$. 

Define $M(t)=\int_0^t\sqrt{X_s}^TdW_s$. Since $X(t)$ is a continuous semimartingale, and $\sqrt{X(t)}$ is
defined and continuous for $0\leq t\leq t_0$, $M(t)$ is a local martingale, and $S(t)$ satisfies the SDE
$$
dS(t) = -\theta(S(t)-1)dt +\sqrt{\frac{2\theta}{\alpha_0}}(1-S(t))dM(t).
$$
Since $X_0\in\s$ a.s., then $S_0=1$ a.s., and $S(t)$ constantly equal to 1 is a solution of the SDE. 
Uniqueness holds for this equation, so we conclude that
$\sum_i X_i(t)=1$ for $0\leq t$, a.s.

In conclusion, the solution $X(t)$ is contained in $\bar\s\subset K$. Since the coefficients of the
modified equation are identical to the coefficients of the original equation~\eqref{canonicalSDE} on $K$, then 
$\left((\Omega,\F,P),\{\F_t\},W(t),X(t)\right)$ is a solution for~\eqref{canonicalSDE}, a.s.
\end{proof}


\begin{thebibliography}{10}

\bibitem{JA82}
J.~Aitchison.
\newblock The statistical analysis of compositional data.
\newblock {\em Journal of the Royal Statistical Society: Series B (Statistical
  Methodology)}, 44(2):139--177, 1982.

\bibitem{balakrishnan2004primer}
N.~Balakrishnan and V.~Nevzorov.
\newblock {\em A Primer on Statistical Distributions}.
\newblock Wiley, 2004.

\bibitem{CaporasoEtAl2011}
J.~Caporaso, C.~Lauber, E.~Costello, D.~Berg-Lyons, A.~Gonzalez, J.~Stombaugh,
  D.~Knights, P.~Gajer, J.~Ravel, N.~Fierer, et~al.
\newblock Moving pictures of the human microbiome.
\newblock {\em Genome Biol}, 12(5):R50, 2011.

\bibitem{KS91}
I.~Karatzas and S.~Shreve.
\newblock {\em Brownian Motion and Stochastic Calculus}.
\newblock Springer-Verlag, 1991.

\bibitem{K12}
R.~Khasminskii.
\newblock {\em Stochastic Stability of Differential Equations}.
\newblock Springer, 2012.

\bibitem{kloeden2011numerical}
P.~Kloeden and E.~Platen.
\newblock {\em Numerical Solution of Stochastic Differential Equations}.
\newblock Stochastic Modelling and Applied Probability. Springer Berlin
  Heidelberg, 2011.

\bibitem{kbj2000}
S.~Kotz, N.~Balakrishnan, and N.~L. Johnson.
\newblock {\em Continuous Multivariate Distributions Volume 1: Models and
  Applications}.
\newblock Wiley-Interscience, New York, 2nd edition, 2000.

\bibitem{Rmanual2024}
{R Core Team}.
\newblock {\em R: A Language and Environment for Statistical Computing}.
\newblock R Foundation for Statistical Computing, Vienna, Austria, 2024.

\bibitem{S88}
C.~Soize.
\newblock Steady-state solution of {F}okker-{P}lanck equation in higher
  dimension.
\newblock {\em Probabilistic Engineering Mechanics}, 3(4):196--206, 1988.

\bibitem{S94}
C.~Soize.
\newblock {\em The Fokker-Planck equation for stochastic dynamical systems and
  its explicit steady state solutions}.
\newblock World Scientific Publishing Co., 1994.

\bibitem{VR-MASS2002}
W.~N. Venables and B.~D. Ripley.
\newblock {\em Modern Applied Statistics with S}.
\newblock Springer, New York, fourth edition, 2002.
\newblock ISBN 0-387-95457-0.

\end{thebibliography}
\end{document}